\documentclass[12pt]{amsart}
\usepackage{amsmath}
\usepackage{amssymb}
\usepackage{amsfonts}
\usepackage{amsthm}
\usepackage{xcolor}
\usepackage{hyperref}
\usepackage{mathrsfs}
\hypersetup{
    bookmarks=true,         % show bookmarks bar?
    colorlinks=true,       % false: boxed links; true: colored links
    linkcolor=blue,          % color of internal links (change box color with linkbordercolor)
    citecolor=green,        % color of links to bibliography
    filecolor=cyan,         % color of file links
    urlcolor=magenta        % color of external links
}

\makeatletter
\@namedef{subjclassname@2020}{2020 Mathematics Subject Classification}
\makeatother

\newtheorem{theo}{Theorem}[section]
\newtheorem*{theo*}{Theorem}
\newtheorem{coro}[theo]{Corollary}
\newtheorem{lemm}[theo]{Lemma}
\newtheorem{prop}[theo]{Proposition}

\newcommand{\C}{\mathbb{C}}

\title[]{On an exponential  power sum}
\begin{document}
 
 \subjclass[2010]{}
 
 \author[N E Thomas]{Neha Elizabeth Thomas}
 \address{Department of Mathematics, University College, Thiruvananthapuram (Research Centre of the University of Kerala), Kerala, India\\ORCID: 0000-0002-5473-5542}
 \email{nehathomas2009@gmail.com}
 \author[K V Namboothiri]{K Vishnu Namboothiri}
\address{Department of Mathematics, Baby John Memorial Government College, Chavara, Sankaramangalam, Kollam, Kerala, INDIA\\Department of Collegiate Education, Government of Kerala, India\\
ORCID: 0000-0003-4516-8117}

\email[Corresponding Author]{kvnamboothiri@gmail.com}
\thanks{}
%\footnotetext[1]{\textit{first author}}
%\footnotetext[2]{\textit{corresponding author}}

 \begin{abstract}
Using combinatorial techniques, we derive a recurrence identity that expresses an exponential power sum with negative powers in terms of another exponential power sum with positive powers. Consequently, we derive a formula for the power sum of the first $k$ natural numbers when the power is odd, which when used in combination with Faulhaber's formula for computing power sums helps us to retrieve the Bernoulli numbers in certain cases.
 \end{abstract}
 
 \maketitle
\section{Introduction}
A Dirichlet character $\chi$ modulo $k$ (where $k$ is a positive integer) is defined to be \textit{odd} if $\chi(-1)=-1$ and \textit{even} if $\chi(-1)=1$. The \textit{Dirichlet $L$-function} $L(s,\chi)$ is defined by the infinite series $\sum\limits_{n=1}^{\infty}\frac{\chi(n)}{n^s}$, where $s\in\C$ with $Re\,(s)>1$.
 The \textit{Gauss sum} $G(z,\chi)$ for any complex number $z$ is defined as

 \begin{align*}
 G(z,\chi):= \sum\limits_{m=1}^{k}\chi(m)e^{\frac{2\pi imz}{k}}.
 \end{align*}

E. Alkan established the following identity connecting $L(r,\chi)$ with Gauss sums and Bernoulli numbers:
  \begin{align}\label{lsidentity}
 \frac{(-1)^{v+1}kr!}{i^r2^{r-1}\pi^r}L(r,\chi)=\sum\limits_{q=0}^{2[\frac{r}{2}]}\binom{r}{q}
 B_q S(r-q,\chi)\text{ \cite[Theorem 1]{alkan2011values}},
 \end{align}
  where
\begin{align*}
S(m, \chi):=\sum\limits_{j=1}^k\left(\frac{j}{k}\right)^mG(j, \chi).
\end{align*}
and $B_q$ are the $q$\textsuperscript{th} Bernoulli numbers. See \cite[Chapter 12]{tom1976introduction} for the definition and other properties of Bernoulli numbers and Bernoulli polynomials.

E. Alkan derived formulae in \cite{alkan2011mean} for the sums $\sum\limits_{\chi \, \text{odd}}|L(1,\chi)|^2$ and $\sum\limits_{\chi \, \text{even}}|L(2,\chi)|^2$ using Identity (\ref{lsidentity}).
His computations involved the simplification of sums of the form $\sum\limits_{s=1}^{k-1}s^pe^{\frac{-2\pi ims}{k}}$. He mainly tried to express such sums in terms of positive arguments of the exponential function. It seems that the complexity of simplifying these exponential power sums prevented him from establishing an exact formula for computing the value of $\sum\limits_{\chi , r\text{ with same parity}}|L(r,\chi)|^2$ in general.

In this paper, we derive  an identity that expresses the exponential power sum $\sum\limits_{s=1}^{k-1}s^pe^{\frac{-2\pi ims}{k}}$ in terms of another exponential power sum with positive powers. Consequently, we derive a recursive formula for the power sum of the first $k$ natural numbers $h(p, k) :=\sum\limits_{s=1}^k s^p$ when $p$ is odd.

We note that Schaumberger and Faulhaber gave two different techniques for finding the power sums of natural numbers. In \cite{Shaumberger1976}, Schaumberger used the exponential generating function $$\sum\limits_{p=0}^{\infty}h(p,k)\frac{x^k}{k!} = \sum\limits_{r=1}^{p} e^{rx}$$ and L'H\^{o}spital's rule for evaluating $h(p,k)$.

 Faulhaber's formula \cite[Identity 1]{mcgown2007generalization}  is a non-recursive formula for finding $h(p, k)$. It involves the Bernoulli numbers $B_s$ and is given by
 \begin{align}\label{Faulhaber}
h(p, k) =\frac{1}{(p+1)}\left(k^{p+1}+\sum\limits_{j=1}^p (-1)^j\binom{p+1}{j}B_j k^{p-j+1}\right).
 \end{align}

In \cite{singh2009defining}, J. Singh defined the power sum $\psi_p(k)= a_1^p+a_2^p+\cdots+a_{\varphi(k)}^p$, where $a_1$, $a_2$, $\cdots$, $a_{\varphi(k)}$ are positive integers each of which is less than or equal to $n$ and relatively prime to $n$. There he gave a relation between $h(p, k)$ and $\psi_p(k)$. In \cite{alkan2014ramanujan}, E. Alkan used the results from \cite{singh2009defining} to obtain a formula for the weighted average $\frac{1}{k^{r+1}}\sum\limits_{j=1}^kj^rc_k(j)$, where $c_k(j)=\sum\limits_{\substack{{m=1}\\(m,k)=1}}^{r}e^{\frac{2 \pi imj}{k}}$
is the Ramanujan sum. Later in \cite{singh2016defining}, J. Singh studied the sum of products of power sums via the multiple products of their exponential generating functions.

 For non integral power sums, some approximate results only seem to exist in the literature. See, for example Theorem 1 in \cite{parks2004sums}.

 \section{The main result and proof}

 We state and prove our main identity  in this section. To begin with, we introduce the terms which we use in this paper. An \textit{ordered partition} or \textit{composition} of a positive integer $n$ is an $m$-tuple  $(j_1,\ldots,j_m)$ with $j_1+\ldots+j_m=n$, where $1\leq j_l\leq n$. The $j_l$ are called \textit{parts} of the composition and $m$, the number of parts in it, is called the \textit{length} of the composition. We use $C(n)$ to denote the set of all compositions of $n$. For $1\leq m\leq n$, $C(m,n)$ will denote the set of all compositions of $n$ with length $m$. The number of elements in $C(m,n)$ is  $\binom{n-1}{m-1}$ (\cite[Theorem 4.1]{andrews1998theory}). More properties and results related to compositions can be found, for example, in \cite[Chapter 4]{andrews1998theory}.

In this paper, we mainly derive the identity given in the next proposition.

\begin{prop}\label{mainprop}
For $k,p$ positive integers with $k\nmid m$ and $p\geq 1$, we have
\begin{align*}\sum\limits_{s=1}^{k-1}s^pe^{\frac{-2\pi ims}{k}}=-k^p + \sum\limits_{a=0}^{p-1}(-1)^{p-a}\binom{p}{a}k^{a}\sum\limits_{s=1}^{k-1}s^{p-a}e^{\frac{2\pi ims}{k}}.
\end{align*}
\end{prop}

We state and prove a lemma which is essential in the proof of the main proposition.

\begin{lemm}\label{mainlemma} Let $r\geq 1$.  The set $\{(i_1,\ldots, i_r): i_1,i_2,\cdots, i_r\in \{1,\cdots ,p-1\}, i_1>i_2>\cdots>i_r\}$ is in one to one correspondence with the set $C(r+1,p)$.
\end{lemm}
\begin{proof}
Write $p-i_1=j_1$, $i_1-i_2=j_2$, $\cdots$, $i_{r-1}-i_r=j_r, i_r-0=j_{r+1}$. Since $j_1+j_2+ \cdots + j_{r+1}=p$ and $j_l$ are all positive, $(j_1,\ldots, j_{r+1})\in C(r+1,p)$. Let $(i_1, i_2, \ldots, i_r)$ and $(i_1',i_2', \ldots, i_r')$ be distinct. Let $l$ be the largest index such that  $i_l\neq i_l'$ (with the convention that $i_{l+1}=0$ if $l=r$). Thus $j_{l+1}=i_l-i_{l+1}$ and $j_{l+1}'=i_l'-i_{l+1}'$ will be distinct. Hence $(j_1, j_2, \ldots, j_{r+1})$ and $(j_1',j_2', \ldots, j_{r+1}')$ will be distinct. On the other hand, if $(j_1, j_2, \ldots, j_{r+1})\in C(r+1,p)$, let $i_1=p-j_1, i_2 = p-(j_1+j_2),\ldots$ so that $i_1,\ldots,i_r \in\{1,\ldots,p-1\}$ and they satisfy $i_1>\ldots>i_r$, with $r=0$ giving the empty set. This shows that this correspondence is onto as well.
\end{proof}
Since $C(r+1,p)$ has $\binom{p-1}{r}$ elements, we obtain the following corollary.
\begin{coro}
 The set defined in the above proposition has $\binom{p-1}{r}$ elements.
 \end{coro}

We state a result from \cite{gessel2013compositions} which we will be using in the proof of our main proposition.

\begin{lemm}(\cite[Lemma 1]{gessel2013compositions})\label{gessellemma}
If $u_1, u_2, u_3, \ldots$ is any sequence of numbers, then the coefficient of $x^n$ in $\left(1-\sum \limits_{i=1}^{\infty}u_ix^i\right)^{-1}$ is $\sum \limits_{(a_1,\ldots,a_k)\in C(n)}u_{a_1}u_{a_2}\cdots u_{a_k}$.
\end{lemm}

Now we prove Proposition \ref{mainprop}.
\begin{proof}[Proof of Proposition \ref{mainprop}]

We start the proof using the binomial expansion:

\begin{align*}
\sum\limits_{s=1}^{k-1}s^pe^{\frac{2\pi ims}{k}}&
=\sum\limits_{s=1}^{k-1}(k-s)^pe^{\frac{2\pi im(k-s)}{k}}\\
&=\sum\limits_{s=1}^{k-1}(k-s)^pe^{-\frac{2\pi ims}{k}}\\
&=\sum\limits_{s=1}^{k-1}(-1)^ps^pe^{-\frac{2\pi ims}{k}}+\sum\limits_{a=0}^{p-1}(-1)^{a}\binom{p}{a}k^{p-a}\sum\limits_{s=1}^{k-1}s^{a}e^{\frac{-2\pi ims}{k}}.
\end{align*}

Hence
\begin{align}
\sum\limits_{s=1}^{k-1}s^pe^{-\frac{2\pi ims}{k}}=(-1)^p\sum\limits_{s=1}^{k-1}s^pe^{\frac{2\pi ims}{k}}+\sum\limits_{a=0}^{p-1}(-1)^{p+a+1}\binom{p}{a}k^{p-a}\sum\limits_{s=1}^{k-1}s^{a}e^{\frac{-2\pi ims}{k}}.\label{idn:negexptopos}
\end{align}
For notational convenience, we write $f(p)=\sum\limits_{s=1}^{k-1}s^pe^{\frac{-2\pi ims}{k}}$ and $g(p)= \sum\limits_{s=1}^{k-1}s^{p}e^{\frac{2\pi ims}{k}}$. With this notation, we rewrite Equation \eqref{idn:negexptopos} and expand it recursively as
\begin{align*}
f(p)=&(-1)^pg(p)+\sum\limits_{a=0}^{p-1}(-1)^{p+a+1}k^{p-a}\binom{p}{a}f(a)\\
=&(-1)^pg(p)+(-1)^{p+1}\binom{p}{p-1}kg(p-1)\\&+\left[(-1)^{p+1}\binom{p}{p-2}+(-1)^{p+2}\binom{p}{p-1}\binom{p-1}{p-2}\right]k^2g(p-2)\\
&+\Bigg[(-1)^{p+1}\binom{p}{p-3}+(-1)^{p+2}\binom{p}{p-2}\binom{p-2}{p-3}\\
&+(-1)^{p+2}\binom{p}{p-1}\binom{p-1}{p-3}+(-1)^{p+3}\binom{p}{p-1}\binom{p-1}{p-2}\binom{p-2}{p-3}\Bigg]\\&\times k^3g(p-3)+\cdots.
\end{align*}
It can be seen by induction that for $a\geq 2$, $g(p-a)$ has coefficient as the sum
\begin{align}
\sum \limits_{\substack{i_1,i_2,\cdots, i_r\in \left\lbrace p-a+1,\cdots ,p-1\right\rbrace\\i_1>i_2>\cdots>i_r}}(-1)^{p+r+1}\binom{p}{i_1}\binom{i_1}{i_2}\ldots \binom{i_{r-1}}{i_r}\binom{i_{r}}{p-a}k^a\label{idn:gpacoeff},
\end{align}
with the convention that $\{i_1,\ldots, i_r\}$ can also be the empty set to allow the binomial coefficient $\binom{p}{p-a}$.
Thus the coefficient of $k^pg(0)$ is
\begin{align}\label{kpg0}
\sum \limits_{\substack{i_1,i_2,\cdots, i_r\in \left\lbrace 1,\cdots ,p-1\right\rbrace\\i_1>i_2>\cdots>i_r}}(-1)^{p+r+1}\binom{p}{i_1}\binom{i_1}{i_2}\ldots \binom{i_{r-1}}{i_r}\binom{i_{r}}{0}.
\end{align}
Now
\begin{align*}
\binom{p}{i_1}\binom{i_1}{i_2}\ldots \binom{i_{r-1}}{i_r}\binom{i_{r}}{0}
&=\frac{p!}{(p-i_1)!i_1!}\frac{i_1!}{(i_1-i_2)!i_2!}\cdots\frac{i_{r-1}!}{(i_{r-1}-i_r)!i_r!}\frac{i_r!}{(i_r-0)!0!}\\
&=\frac{p!}{(p-i_1)!(i_1-i_2)!\cdots(i_r-0)!0!}.
\end{align*}
Now we use Lemma \ref{mainlemma}.  The binomial product sum in Equation (\ref{kpg0}) consists of the summand $(-1)^{p+1}\binom{p}{0}$. This term corresponds to the empty subset of $\{1,\ldots,p-1\}$. Corresponding to this empty set we may take the single available element $(p)$ from $C(1,p)$.
Hence, $\binom{p}{i_1}\binom{i_1}{i_2}\cdots \binom{i_{r-1}}{i_r}\binom{i_{r}}{0}$ and $\binom{p}{0}$ can be replaced with the multinomial coefficient $\binom{p}{j_1, j_2, \cdots , j_{r+1}}$  and so the coefficient of $k^p g(0)$ can be rewritten as
\begin{align*}
\sum\limits_{r=0}^{p-1}\sum\limits_{\substack{(j_1,j_2,\ldots,j_{r+1})\in C(r+1,p)}}&(-1)^{p+r+1} \binom{p}{j_1, j_2, \cdots , j_{r+1}}\\
&=\sum\limits_{r=1}^{p}\sum\limits_{\substack{(j_1,j_2,\ldots,j_{r})\in C(r,p)}}(-1)^{p+r} \binom{p}{j_1, j_2, \cdots , j_{r}}\\
&=(-1)^p p!\sum\limits_{r=1}^{p}\sum\limits_{\substack{(j_1,j_2,\ldots,j_{r})\in C(r,p)}}(-1)^{r} \frac{1}{j_1! j_2! \cdots  j_{r}!}\\
&=(-1)^p p!\times \text{ coefficient of } x^p \text{ in }  \left(1+\sum\limits_{i=1}^\infty\frac{x^i}{i!}\right)^{-1}\\&   \left(\text{obtained by taking } u_i=\dfrac{-1}{i!}\text{  in Lemma \ref{gessellemma}}\right)\\
&=(-1)^p p!\times \text{ coefficient of } x^p \text{ in } \left(e^{x}\right)^{-1}\\
&= 1.
\end{align*}
Now we compute the coefficient of $k^ag(p-a)$. We take the product of binomial coefficients appearing in Equation \eqref{idn:gpacoeff} and simplify it further:
\begin{align*}
\binom{p}{i_1}\binom{i_1}{i_2}\ldots \binom{i_{r-1}}{i_r}\binom{i_{r}}{p-a}
&=\binom{p}{p-i_1}\binom{i_1}{i_1-i_2}\ldots \binom{i_{r-1}}{i_{r-1}-i_r}\binom{i_{r}}{i_r-p+a}\\
&=\frac{p!}{(p-i_1)!(i_1-i_2)!\cdots(i_r-p+a)!(p-a)!}\\
&=\frac{a!}{(p-i_1)!(i_1-i_2)!\cdots(i_r-p+a)!}\binom{p}{a}.
\end{align*}
Write $p-i_1=j_1$, $i_1-i_2=j_2$, $\cdots, i_{r-1}-i_r=j_r, i_r-(p-a)=j_{r+1}$. So $j_1+j_2+ \cdots + j_{r+1}=a$. As in the above case, we can see that $(j_1,j_2, \cdots, j_{r+1})$ is in a one to one correspondence with  $(i_1$, $i_2$, $\cdots$, $i_r$), where $i_1>\ldots>i_r$. So we have $$\binom{p}{i_1}\binom{i_1}{i_2}\cdots \binom{i_{r-1}}{i_r}\binom{i_{r}}{p-a}=\binom{p}{a} \binom{a}{j_1, j_2, \cdots , j_{r+1}}.$$ Hence we get the coefficient of $k^ag(p-a)$ as
\begin{align*}
(-1)^{p-a}&\sum\limits_{r=0}^{p-1}\sum\limits_{\substack{(j_1,j_2,\ldots,j_{r+1})\in C(a,r+1)}}(-1)^{a+r+1}\binom{p}{a} \binom{a}{j_1, j_2, \cdots , j_{r+1}}\\
&=(-1)^{p-a}\binom{p}{a}\sum\limits_{r=0}^{p-1}\sum\limits_{\substack{(j_1,j_2,\ldots,j_{r+1})\in C(a,r+1)}}(-1)^{a+r+1} \binom{a}{j_1, j_2, \cdots , j_{r+1}}\\
&=(-1)^{p-a}\binom{p}{a}.
\end{align*}
Thus
\begin{align*}
f(p)&=(-1)^pg(p)+(-1)^{p-1}\binom{p}{p-1}kg(p-1)+\sum\limits_{a=2}^{p}(-1)^{p-a}\binom{p}{a}k^{a}g(p-a)\\
&=\sum\limits_{a=0}^{p-1}(-1)^{p-a}\binom{p}{a}k^{a}g(p-a)+k^{p}g(0)\\
&=-k^p + \sum\limits_{a=0}^{p-1}(-1)^{p-a}\binom{p}{a}k^{a}g(p-a)
\end{align*}
which is what we claimed. In the last step above, we used the fact that $g(0)= \sum\limits_{s=1}^{k-1}e^{\frac{2\pi ims}{k}}=-1$  when $e^{\frac{2\pi ims}{k}}\neq 1$.
\end{proof}

Let us see an application of the computations above. When $k|m$, $e^{\frac{2\pi ims}{k}}=1$ and $g(0)=k-1$ in the above so that we get
\begin{align}\label{powersum}
\sum\limits_{s=1}^{k-1}s^p=k^p(k-1) + \sum\limits_{a=0}^{p-1}(-1)^{p-a}\binom{p}{a}k^{a}\sum\limits_{s=1}^{k-1}s^{p-a}.
\end{align}

Let $h(p,k) = \sum\limits_{s=1}^{k}s^p$. If $p$ is odd, from the above, we get a recursive formula for the sum of $p$\textsuperscript{th} powers of the first $k$ natural numbers as

\begin{align*}h(p,k)=\frac{1}{2}\left((k+1)^pk + \sum\limits_{a=1}^{p-1}(-1)^{p-a}\binom{p}{a}(k+1)^{a}h(p-a,k)\right).
\end{align*}

On replacing $p-a$ with $j$, we get the following.

\begin{coro}\label{hpk}For $p,k$ positive integers with $p$ odd, we have
\begin{align}\label{hpkodd}
h(p,k)=\frac{1}{2}\left((k+1)^pk + \sum\limits_{j=1}^{p-1}(-1)^{j}\binom{p}{j}(k+1)^{p-j}h(j,k)\right).
\end{align}

\end{coro}
Such a formula cannot be derived when $p$ is even because of the cancellation of $h(p,k)$ taking place from both sides of Identity (\ref{powersum}).

To illustrate the use of the above identity, let us compute  $\sum\limits_{s=1}^{k}s^3$. Note that the knowledge of  $h(1,k)=\sum\limits_{s=1}^{k}s$ and $h(2,k)=\sum\limits_{s=1}^{k}s^2$ is necessary for this computation. We have $h(1, k)= \frac{k(k+1)}{2}$ and $h(2, k)= \frac{k(k+1)(2k+1)}{6}$.
Thus
\begin{align*}
h(3, k) &= \frac{1}{2}\left( (k+1)^3k + \sum\limits_{a=1}^{2}(-1)^{3-a}\binom{3}{a}(k+1)^a h(3-a,k)\right)\\
&=\frac{1}{2}\left( (k+1)^3k + (-1)^{2}\binom{3}{1}(k+1) h(2,k)+(-1)^{1}\binom{3}{2}(k+1)^2 h(1,k)\right)\\
&=\frac{1}{2}\left( (k+1)^3k + 3(k+1) \frac{k(k+1)(2k+1)}{6}-3(k+1)^2 \frac{k(k+1)}{2}\right)\\
&=\frac{1}{4}\left(   k(k+1)^2(2k+1)-(k+1)^3k \right)\\
&=\left[\frac{k(k+1)}{2}\right]^2.
\end{align*}
We cannot compute $h(4,k)$ using Identity (\ref{hpkodd}). But if we use the known formula   $h(4,k)= \frac{k(k+1)(2k+1)(3k^2+3k-1)}{30}$, we get $h(5,k)$ as $\frac{k^2(k+1)^2(2k^2+2k-1)}{12}$ if we proceed as in the above computation for $h(3,k)$.

Another application of Identity (\ref{hpkodd}) is that we can retrieve the Bernoulli numbers using Identities (\ref{Faulhaber}) and (\ref{hpkodd}). For example, if we put $p=1$ in the right-hand side of both the Equations (\ref{Faulhaber}) and (\ref{hpkodd}) and  equate them, we get
\begin{align*}
\frac{k(k+1)}{2}=\frac{1}{2}\left(k^2+(-1)\binom{2}{1}B_1k \right)= \frac{1}{2}\left(k^2-2B_1k \right).
\end{align*}
On comparing the coefficients of $k$ on both sides, we get $B_1=-\frac{1}{2}$.

Similarly, if we put $p=3$ in the right-hand side of both identities (\ref{Faulhaber}) and (\ref{hpkodd}), we get
\begin{align*}
 \frac{k^4}{4}-B_1k^3+\frac{3}{2}B_2k^2-4B_3k=&\frac{(k+1)^3k}{2}-\frac{3}{2}(k+1)^2\frac{k(k+1)}{2}\\&+\frac{3}{2}(k+1)\frac{k(k+1)(2k+1)}{6}.
\end{align*}
By comparing the coefficients of $k^2$ on both sides in the above, we get $B_2=\frac{1}{6}$.
 Proceeding like this we can find the values of $B_4$, $B_6$,$\cdots$. Note that $B_{2n-1}=0$ for all positive integers $n$ except when $n=1$.

\section{Acknowledgements}
 The first author thanks the Kerala State Council for Science,Technology and Environment, Thiruvananthapuram, Kerala, India for providing financial support for carrying out this research work. The authors thank the reviewer for his/her
comments which helped improve this paper.
 
%\bibliographystyle{integers}
%\bibliography{references}

\end{document}